\documentclass{article}[11]
\usepackage{amsmath, amssymb, color}

\newtheorem{thm}{Theorem}[section]

\newtheorem{cor}[thm]{Corollary}

\newenvironment{example}{\medskip\refstepcounter{thm}
\noindent{\bf Example \thesection.\arabic{thm}\ }}{\medskip}

\newenvironment{proof}[1][,]{\medskip\ifcat,#1
\noindent{{\it Proof}:\ \ }\else\noindent{\it Proof of #1.\ }\fi}
{\hfill$\square$\medskip}
\newenvironment{remark}[1][Remark]{\begin{trivlist}
\item[\hskip \labelsep {\bfseries #1}]}{\end{trivlist}}

\DeclareMathOperator\vol{vol}

\DeclareMathOperator\gl{gl}

\DeclareMathOperator\tR{\tilde{R}}

\DeclareMathOperator\tnabla{\widetilde{\nabla}}

\DeclareMathOperator\tr{tr}
\def\C{\mathbb{C}}

\def\Sph{\mathbb{S}}

\def\R{\mathbb{R}}

\begin{document}

\title{Maslov, Chern-Weil and Mean Curvature}

\author{Tommaso Pacini\footnote{University of Torino, Italy, \texttt{tommaso.pacini@unito.it}}}


\maketitle

\begin{abstract}
We provide an integral formula for the Maslov index of a pair $(E,F)$ over a surface $\Sigma$, where $E\rightarrow\Sigma$ is a complex vector bundle and $F\subset E_{|\partial\Sigma}$ is a totally real subbundle. As in Chern-Weil theory, this formula is written in terms of the curvature of $E$ plus a boundary contribution. 

When $(E,F)$ is obtained via an immersion of $(\Sigma,\partial\Sigma)$ into a pair $(M,L)$ where $M$ is K\"ahler and $L$ is totally real, the formula allows us to control the Maslov index in terms of the geometry of $(M,L)$. We exhibit natural conditions on $(M,L)$ which lead to bounds and monotonicity results.
\end{abstract}
\section{Introduction}\label{s:intro}
The goal of this paper is to investigate and generalize a rather surprising relationship between two apparently unrelated quantities: the Maslov index of a surface with boundary, and the mean curvature of its boundary constraint. 

Some forms of this relationship are already known, cf. \textit{e.g.} \cite{Morvan} concerning Lagrangian boundary data in $\C^n$ or \cite{CG} concerning surfaces $\Sigma$ immersed in a K\"ahler-Einstein manifold $M$ whose boundaries $\partial\Sigma$ lie on a minimal Lagrangian submanifold $L\subset M$. 

These geometric assumptions on $(M,L)$ are, however, extremely strong. The Maslov index is well-defined even with respect to non-integrable complex structures on $M$ and totally real boundary data, while mean curvature can be defined with respect to any metric. It is thus interesting to understand if such a relationship exists in a more general context.

Our tool for studying this issue is a very general integral formula for the Maslov index of an abstract ``bundle pair'' $(E,F)$ over a surface with boundary $(\Sigma,\partial\Sigma)$, cf. Equation (\ref{eq:MCW}). Standard definitions of this index are either topological or axiomatic. The integrand in our formula involves data arising from a connection on $E$. It thus provides a third, geometric, definition in the spirit of Chern-Weil theory. 

When the pair is determined by an immersion of $(\Sigma,\partial\Sigma)$ into $(M,L)$ we can rephrase this formula in terms of geometric data on $(M,L)$, cf. Equation (\ref{eq:MCWbis}). This formula holds under very general assumptions on $(M,L)$. In the special case where $M$ is K\"ahler and $L$ is totally real we show that the boundary contribution can be expressed in terms of the ``$J$-mean curvature'' of $L$, defined in terms of a perturbed volume functional specifically tailored towards totally real submanifolds, introduced in \cite{Bor}. Finally, when $M$ is K\"ahler and $L$ is Lagrangian we recover the classical expression of the boundary contribution in terms of the standard mean curvature of $L$. Going in the opposite direction, we thus find that the ``root cause'' of the appearance of mean curvature is the same as that which generates the whole body of differential geometry of totally real submanifolds developed in \cite{LP}, \cite{LP2}, \cite{LP3}.

Applications of our formulae are discussed in Section \ref{s:applications}. We show that they provide a unifying point of view on several results in the literature, also extending them to the more general context of non-constant curvature, totally real boundary data or non-integrable $J$.

\ 

\textit{Acknowledgements.}
My interest in this topic was largely triggered by stimulating conversations with Claude LeBrun, Jason Lotay and Roberta Maccheroni.

Integral formulae for certain Maslov indices also appear for example in \cite{Dazord}, \cite{Ono} and \cite{CS}. These papers however mostly focus on pairs determined by immersions of $\Sigma$ into a K\"ahler manifold $M$, and only consider Lagrangian boundary data. Only the formula in \cite{CS} allows for abstract bundle pairs, but it relies on a simplifying hypothesis which completely eliminates the (Lagrangian) boundary contribution. Our own formula includes all these cases, but is much more general. In summary, results concerning surfaces immersed in $M$ K\"ahler and with boundary on $L$ Lagrangian should probably be considered classical. Beyond this case, we are not aware of serious overlaps between this paper and the available literature.

\section{Preliminaries}\label{s:prelim}
\paragraph{Induced curvatures.} Let $E$ be a $k$-dimensional complex bundle over any manifold $N$, possibly with boundary. We will typically denote sections of $E$ by $\sigma$, sections of $E^*$ by $\alpha$ and tangent vectors or vector fields on $N$ by $X$.

Choose a connection $\nabla$ on $E$, thus $\nabla:\Lambda^0(E)\rightarrow \Lambda^1(E)$. As usual, we can extend $\nabla$ to all tensor bundles associated to $E$ via the Leibniz rule and by imposing that the connections commute with contractions. We will denote these induced connections using the same symbol $\nabla$. In particular,
\begin{itemize}
 \item The induced connection on $E^*$, equivalently on $\overline{E}$, satisfies
 \begin{equation*}
(\nabla_X\alpha)\sigma=X(\alpha(\sigma))-\alpha(\nabla_X\sigma).  
 \end{equation*}
\item The induced connection on $\Lambda^kE:=E\wedge\dots\wedge E$ satisfies
\begin{equation*}
\nabla_X(\sigma_1\wedge\dots\wedge\sigma_k)=\left((\nabla_X\sigma_1)\wedge\dots\wedge\sigma_k\right)+\dots+\left(\sigma_1\wedge\dots\wedge(\nabla_X\sigma_k)\right).
\end{equation*}
\end{itemize}
Let $R\in\Lambda^2(N)\otimes\gl(E)$ denote the curvature of $\nabla$, defined by
\begin{equation*}
 R(X,Y)\sigma:=\nabla_X\nabla_Y\sigma-\nabla_Y\nabla_X\sigma-\nabla_{[X,Y]}\sigma.
\end{equation*}
One can use the analogous formula to calculate the curvatures of the induced connections. In particular,
\begin{itemize}
 \item The curvature on $E^*$ is $-R^*$, where $R^*\in\Lambda^2(N)\otimes\gl(E^*)$ is the dual of $R$.
 \item The curvature on $\Lambda^kE$ is $\tr(R)\in\Lambda^2(N)\otimes\gl(\Lambda^kE)$.
\end{itemize}
Using the identifications $\gl(\Lambda^kE)\simeq\C\simeq\gl(\Lambda^kE^*)$ we conclude that the curvature on $\Lambda^kE^*$ is $-\tr(R^*)=-\tr(R)\in\Lambda^2(N)$.

Any other connection on $E$ is of the form $\tnabla=\nabla+A$, for some $A\in\Lambda^1(N)\otimes \gl(E)$. The corresponding curvature is $\tR=R+dA+A\wedge A$. On the induced bundles one obtains
\begin{itemize}
 \item On $E^*$, the connections are related by $\tnabla=\nabla-A^*$ where $A^*\in\Lambda^1(N)\otimes\gl(E^*)$ is the dual of $A$. The corresponding curvature is $-R^*-d(A^*)+A^*\wedge A^*$.
 \item On $\Lambda^kE$, $\tnabla=\nabla+\tr(A)$, where $\tr(A)\in\Lambda^1(N)\otimes\gl(\Lambda^kE)\simeq\Lambda^1(N)$. The corresponding curvature is $\tr(R)+d(\tr(A))$ because $\tr(A)\wedge\tr(A)=0$.
 \item On $\Lambda^kE^*$, $\tnabla=\nabla-\tr(A^*)=\nabla-\tr(A)$. The corresponding curvature is $-\tr(R)-d(\tr(A))$.
\end{itemize}

Recall from Chern-Weil theory that $\tr(R)\in\Lambda^2(N)$ is closed. It follows from the above that it is independent of $\nabla$ up to an exact form $d(\tr(A))$. Choosing $\nabla$ to be unitary with respect to some choice of Hermitian metric $h$ on $E$ we obtain $\tr(R)\in i\Lambda^2(N,\R)$.

\paragraph{The Maslov index.} Let $(\Sigma,\partial\Sigma)$ be a compact surface with (possibly empty) boundary. Let $E$ be a $k$-dimensional complex vector bundle over $\Sigma$ and $F\subset E$ a $k$-dimensional totally real subbundle of $E_{|\partial\Sigma}$, defined over $\partial\Sigma$.
\begin{itemize}
\item When $\partial\Sigma\neq\emptyset$ (thus a collection of circles) it is well-known that $E$ is trivial, though not in a unique way. Given a choice of trivialization, one can measure the twisting (winding number) of $F$ around $\partial\Sigma$ with respect to the trivialization of $E$, obtaining an integer. This is a topological definition of the \textit{Maslov index} of the pair $(E,F)$. One can check that it is independent of the choice of trivialization. This index can alternatively be characterized axiomatically in terms of its behaviour under natural ``connect sum'' operations on $(E,F)$ and $(\Sigma,\partial\Sigma)$: we refer to \cite{MS} Appendix C.3 for details.
\item When $\partial\Sigma=\emptyset$ (thus $F$ is not defined) the Maslov index coincides with twice the first Chern number $c_1(E)\cdot\Sigma$ of $E$.
\end{itemize}
The Maslov index appears in many different contexts in geometry and analysis. In particular, when $\Sigma$ is a Riemann surface, it appears in the generalized Riemann-Roch theorem for surfaces with boundary, cf. \cite{MS} Appendix C.1: it then replaces the term $c_1(E)\cdot\Sigma$ (\textit{i.e.} the degree of $E$) which appears in the classical Riemann-Roch theorem concerning holomorphic line bundles $E$ over closed surfaces. In both cases, an important geometric application of this theorem concerns the deformation theory (moduli spaces) of Riemann surfaces immersed in complex manifolds $(M,J)$, where $J$ is not necessarily integrable. In this case one chooses $E$ to be $TM_{|\Sigma}$ (alternatively, to prevent reparametrizations, $E$ may be the complex normal bundle over $\Sigma$). When $\Sigma$ has boundary, the theory is well-defined as long as $\partial\Sigma$ is constrained to lie on a given totally real submanifold $L\subset M$: the tangent space of $L$ then defines the subbundle $F$. Bounds on the Maslov index give information on the dimension of the moduli space.

\paragraph{Monotonicity.} Deformation theory plays an important role in Symplectic Geometry. In this case we are given a symplectic manifold $(M,\omega)$ and we choose a (not necessarily integrable) compatible complex structure $J$. In this setting there is an important subclass of totally real manifolds: \textit{Lagrangian submanifolds}, defined by the condition $\omega_{|TL}\equiv 0$. When $\partial\Sigma$ lies on a Lagrangian submanifold, it is interesting to compare the Maslov index with the quantity $\int_\Sigma\omega$: the Lagrangian submanifold $L$ is \textit{monotone} if these quantities are proportional (with a fixed constant), for any such $\Sigma$. Monotonicity implies a very strong control on the Maslov index, thus on the deformation theory of surfaces with boundary on $L$: this has important consequences in Floer theory.

\section{A Chern-Weil formula for the Maslov index}\label{s:CWformula}

To simplify we restrict our attention to the oriented case. Specifically, let $(\Sigma,\partial\Sigma)$ be an oriented compact surface with boundary and $(E,F)$ be a bundle pair as in Section \ref{s:prelim}, with $\mbox{dim}_{\C}(E)=k$ and $F$ orientable.

Fix $\nabla$ on $E$. Choosing alternatively $N:=\Sigma$ or $N:=\partial\Sigma$ we can consider as in Section \ref{s:prelim} the induced connections on both $\Lambda^kE\rightarrow\Sigma$ and on $\Lambda^kE\rightarrow \partial\Sigma$. In both cases the bundle is trivial, though not uniquely: different choices are parametrized by maps $N\rightarrow\Sph^1$. 

In the latter case, however, the subbundle $F$ provides a canonical choice (up to homotopy) of a trivialization, as follows. Choose a hermitian metric $h$ on $E$ and an orientation of $F$. For any point $x\in\partial\Sigma$, define
\begin{equation}\label{eq:sigmaJ}
 \sigma_J:=\frac{\sigma_1\wedge\dots\wedge\sigma_k}{|\sigma_1\wedge\dots\wedge\sigma_k|_h},
\end{equation}
where $\sigma_1\dots\sigma_k$ is any positively-oriented (real) basis of the fiber $F_x$. It is clear that this definition is independent of the choice of basis. Choosing a different metric $\tilde{h}$ we obtain a section $\tilde{\sigma}_J=f\sigma_J$, for some $f:\partial\Sigma\rightarrow \R^+$, proving that $\sigma_J$ is well-defined up to homotopy. 

The connection on $\Lambda^kE\rightarrow\partial\Sigma$ is thus completely defined by the \textit{connection 1-form} $\theta\in\Lambda^1(M)$ defined by the identity \begin{equation*}
\nabla\sigma_J=\theta\otimes\sigma_J                                                                                                                                                                      
\end{equation*}
Changing metric gives $\nabla\tilde{\sigma}_J=\tilde{\theta}\otimes\tilde{\sigma}$, where $\tilde{\theta}=\theta+df$. Changing the orientation on $F$ gives $-\sigma_J$, thus the same connection 1-form.

\paragraph{The main formula.} We combine the curvature of $\Lambda^kE$ on $\Sigma$ with the connection 1-form on $\partial\Sigma$ to define the number
\begin{equation}\label{eq:MCW}
 \mu(E,F):=\frac{i}{\pi}\left(\int_\Sigma \tr(R)-\int_{\partial\Sigma}\theta\right).
\end{equation}
This is independent of $\theta$, \textit{i.e.} of $h$ thus of $\sigma_J$, because $\int_{\partial\Sigma}\tilde{\theta}=\int_{\partial\Sigma}\theta$. It is also independent of the choice of $\nabla$ because, using $\nabla+A$, the induced curvature on $\Sigma$ changes to $\tr(R)+d(\tr(A))$ while the connection 1-form becomes $\theta+\tr(A)$: the additional terms then cancel via Stokes' theorem. Choosing $\nabla$ unitary with respect to $h$ we find $\mu(E,F)\in\R$.

Notice that the pair $(E,F)$ induces a pair $(\Lambda^kE,\Lambda^kF)$ on $(\Sigma,\partial\Sigma)$. It is clear from the definition that $\mu(E,F)=\mu(\Lambda^kE,\Lambda^kF)$. Furthermore $\mu(\overline{E},F)=-\mu(E,F)$ because the induced connection on $\overline{E}\simeq E^*$ has the opposite sign. Finally, $\mu(E,F)$ changes sign if we change the  orientation of $\Sigma$. 

\begin{example} \label{ex:disk}
Let $\Delta$ denote the closed unit disk in $\C$. Let $(\Sigma,\partial\Sigma):=(\Delta,\Sph^1)$, let $E:=\C$ be the trivial bundle and $F:=T\Sph^1$ be the tangent bundle to the boundary. Choose the trivial (flat) connection and the standard metric $h$. Then $\sigma_J=\partial\psi$, the standard unit vector field along $\Sph^1$, and $\nabla\sigma_J=id\psi\otimes\sigma_J$ so the connection 1-form is $\theta=id\psi$. It follows that $\mu(E,F)=2$. 
\end{example}

\begin{remark}
When $\partial\Sigma=\emptyset$, (\ref{eq:MCW}) coincides with the standard Chern-Weil formula for $2c_1(E)\cdot\Sigma$.
\end{remark}

\begin{thm} The number $\mu(E,F)$ defined in (\ref{eq:MCW}) coincides with the Maslov index of $(E,F)$. In particular it is integer-valued.
\end{thm}
\begin{proof}
 It is simple to check that $\mu(E,F)$ satisfies the axioms listed in \cite{MS} Appendix C.3. Example \ref{ex:disk} verifies, for example, the normalization axiom.
\end{proof}

\section{The formula for immersed surfaces}\label{s:immersed}

Let $(M,J)$ be a real $2n$-dimensional manifold endowed with a (not necessarily integrable) complex structure $J$. Let $L\subset M$ be an oriented immersed $n$-dimensional totally real submanifold. In this section we assume that $\Sigma$ is immersed in $M$ in such a way that $\partial\Sigma$ is immersed in $L$. We then obtain canonical data $E:=TM_{|\Sigma}$, the pullback tangent bundle, and $F:=TL_{|\partial\Sigma}$, the totally real subbundle defined by the tangent bundle of $L$. Notice that $\Lambda^nE$ coincides with the anti-canonical bundle of $M$, \textit{i.e.} the dual of $K_M$. 

We are interested in computing $\mu_L(\Sigma,\partial\Sigma):=\mu(E,F)$. Fix a Hermitian metric $h$ on $M$ and a unitary connection $\nabla$. In \cite{LP} it is shown that ${K_M}_{|L}$ is trivial and admits a canonical section $\Omega_J$. The corresponding connection 1-form is an element of $i\Lambda^1(L,\R)$: writing it as $i\xi_J$ we obtain a real-valued 1-form $\xi_J$ on $L$ called the \textit{Maslov 1-form}. Restricted to $\partial\Sigma$, $\Omega_J$ is the dual of the section $\sigma_J$ defined in (\ref{eq:sigmaJ}) so the corresponding connection 1-forms differ by a sign.

Let $P$ denote the 2-form on $M$ defined, at any point $x\in M$, by
\begin{equation*}
 P(X,Y):=\omega(R(X,Y)e_i,e_i),
\end{equation*}
where $R$ is the curvature of $\nabla$ and $e_1\dots e_{2n}$ is an orthonormal basis of $T_xM$. One can check that $P=2i\tr(R)$, so general theory ensures that $d\xi_J=\frac{1}{2}P_{|TL}$.

It follows that, in this context, we can re-write the Maslov index in terms of the geometry of $(M,L)$:
\begin{equation}\label{eq:MCWbis}
\mu_L(\Sigma,\partial\Sigma)=\frac{1}{2\pi}\int_\Sigma P-\frac{1}{\pi}\int_{\partial\Sigma}\xi_J.
\end{equation}

\begin{cor}
 If $\Sigma_1$ and $\Sigma_2$ belong to the same homology class in $H_2(M,L)$ then $\mu_L(\Sigma_1,\partial\Sigma_1)=\mu_L(\Sigma_2,\partial\Sigma_2)$.
\end{cor}
\begin{proof}
 This is a standard fact which can be proved using the topological definition of $\mu_L$. Alternatively, let $T$ be a 3-dimensional submanifold in $M$ such that $\partial T= \Sigma_1-\Sigma_2-\Sigma$, with $\Sigma\subseteq L$ and $\partial \Sigma=\partial\Sigma_1-\partial\Sigma_2$. Then $\frac{1}{2}\int_{\Sigma}P=\int_{\Sigma}d\xi_J$ so the result follows from Equation (\ref{eq:MCWbis}) and Stokes' theorem.
\end{proof}

The Maslov indices of such surfaces can thus be collected into a single relative cohomology class $\mu_L\in H^2(M,L)$, known as the \textit{Maslov class} of $L$.

When $M$ is K\"ahler one can check that $P(X,Y)=2\rho(X,Y)$, where the latter is the standard \textit{Ricci 2-form} $\rho(X,Y):=\mbox{Ric}(JX,Y)$. Furthermore, for any totally real submanifold we prove in \cite{LP} the fundamental relationships 
\begin{equation}\label{eq:K_fundamental}
\xi_J=\omega(H_J,\cdot)_{|TL},\ \ d\xi_J=\rho_{|TL},
\end{equation}
where $H_J$ is the negative gradient of the \textit{$J$-volume functional} defined by integrating $\Omega_J$ over $L$. If $L$ is Lagrangian the $J$-volume coincides with the standard Riemannian volume and $H_J$ coincides with the standard mean curvature vector field $H$ of $L$. 

\begin{cor}\label{cor:kahler_formula} Assume $(M,J,\omega)$ is K\"ahler with Ricci 2-form $\rho$. The Maslov index of $(\Sigma,\partial\Sigma)\subset (M,L)$ has the following integral representation:
 \begin{itemize}
\item For $L$ totally real with $J$-mean curvature $H_J$,
\begin{equation*}
 \mu_L(\Sigma,\partial\Sigma)=\frac{1}{\pi}\left(\int_\Sigma \rho-\int_{\partial\Sigma}\omega(H_J,\cdot)\right).
\end{equation*}
\item For $L$ Lagrangian with mean curvature $H$,
\begin{equation*}
 \mu_L(\Sigma,\partial\Sigma)=\frac{1}{\pi}\left(\int_\Sigma \rho-\int_{\partial\Sigma}\omega(H,\cdot)\right).
\end{equation*}
\end{itemize}
\end{cor}
The Lagrangian case already appears, in implicit form, in \cite{CG}. Notice that in this case all terms in the Chern-Weil formula (\ref{eq:MCW}) for the Maslov index become classical geometric quantities.

In order to better understand the quantity $H_J$, let us recall the standard definition: a submanifold $L$ is \textit{minimal} if it is a critical point of the standard Riemannian volume, \textit{i.e.} $H=0$. Analogously, we say that a totally real submanifold $L$ is \textit{$J$-minimal} if it is a critical point of the $J$-volume functional, \textit{i.e.} $H_J=0$; equivalently, $\xi_J=0$ or $\Omega_J$ is parallel.

Equation (\ref{eq:K_fundamental}) shows that any minimal Lagrangian must automatically satisfy the additional condition $\rho_{|TL}\equiv 0$. Coupled with the Lagrangian condition $\omega_{|TL}\equiv 0$, the resulting system typically does not admit solutions unless $M$ is K\"ahler-Einstein: in this case the two conditions coincide. The notion of $J$-minimal submanifolds provides an interesting extension of the minimal Lagrangian condition to the more general setting of K\"ahler manifolds: they are again defined variationally and Equation (\ref{eq:K_fundamental}) again shows that they satify the condition $\rho_{|TL}\equiv 0$, but by definition they are only totally real, not necessarily Lagrangian. In \cite{LP3} we show that $J$-minimal submanifolds also have several analytic properties in common with minimal Lagrangians, and provide examples.

\begin{remark}
 When $M$ is not K\"ahler, torsion terms appear in gradient of the $J$-volume functional and in Equation (\ref{eq:K_fundamental}). We refer to \cite{LP} for details.
\end{remark}

\section{Applications}\label{s:applications}

We now list a few applications of the above formulae.

\paragraph{Comparison with the Gauss-Bonnet theorem.} Recall the standard Gauss-Bonnet theorem: given a smooth domain $U$ in an orientable surface $M$ with metric $g$, 
\begin{equation*}
 \int_U K \vol_g+\int_{\partial U} k \,ds=2\pi\,\chi(U),
\end{equation*}
where $K$ is the Gaussian curvature of $M$ and $k$ is the geodesic curvature of the boundary which, using an arclength parametrization $\gamma(s)$ of $\partial U$ and normal vector field $n(s)$, is defined by the identity $\nabla\dot\gamma=k\,ds\otimes n(s)$. 

For dimensional reasons such $M$ is automatically K\"ahler and $\partial U$ is trivially Lagrangian. Let us choose $(\Sigma,\partial\Sigma)=(\overline{U}, \partial U)$. Then $\sigma_J=\dot\gamma$ and $n=J\dot\gamma$ so $\theta=ik\,ds$. This shows that, up to a factor $\pi$, (\ref{eq:MCW}) coincides with the left-hand side of the Gauss-Bonnet formula. We conclude that $\mu_L(\Sigma,\partial\Sigma)=2\chi(U)$, generalizing Example \ref{ex:disk}.

\paragraph{Vanishing conditions.} Natural assumptions on $(M,L)$ ensure the vanishing of one or both terms in the formula for $\mu_L$:
\begin{itemize}
 \item Assume $P\equiv 0$. This happens for example when $M$ is K\"ahler Ricci-flat, \textit{e.g.} when $M$ is $\C^n$ or, more generally, Calabi-Yau. Then $\mu_L$ depends only on the boundary contribution.
 \item Assume $L$ is such that $\Omega_J$ is parallel: this is similar to the ``orthogonality'' condition appearing in \cite{CS}, but in our context it is geometrically motivated by the notion of $J$-minimal submanifolds. In this case the boundary contribution vanishes and the Maslov index is completely determined by $P$.
\end{itemize} 

As an application of the first vanishing statement we generalize a result of \cite{Dazord} from Lagrangian to totally real boundary data.
\begin{cor}\label{cor:integral_H}
 Assume $M$ is K\"ahler Ricci-flat. Then $\int_{\partial\Sigma}\omega(H_J,\cdot)$ is a multiple of $\pi$, for any totally real $L$ and any $(\Sigma,\partial\Sigma)\subset (M,L)$. 
\end{cor}
For example, when $M=\C^n$ we can produce a 1-parameter family of totally real submanifolds $tL$ simply by rescaling. One can check that $H_J(t)=\frac{1}{t}H_J$, so the boundary contribution for $t\Sigma$ is constant. On the other hand Maslov indices are integral, thus also independent of $t$: this is consistent with Corollary \ref{cor:integral_H}.

The following result uses the second vanishing condition.

\begin{cor}\label{cor:definite_ricci}
 Let $M$ be K\"ahler. Assume the Ricci 2-form $\rho$ has a semi-definite sign (\textit{e.g.}, $\rho\leq 0$). Then, for any minimal Lagrangian $L$, $\mu_L(\Sigma,\partial\Sigma)$ has that same sign for any complex $\Sigma$ (\textit{e.g.}, $\mu_L\leq 0$). 
 
 More generally this holds for any $J$-minimal submanifold.
 \end{cor}
The same statement holds also when $J$ is non-integrable, as long as one replaces $\rho$ with $P$ and $L$ satisfies $\nabla\Omega_J\equiv 0$.

There are also interesting situations where both terms vanish, leading to Maslov-zero submanifolds. For example, (\ref{eq:K_fundamental}) shows that $\xi_J$ is closed when $M$ is K\"ahler Ricci-flat. We thus obtain the following result.

\begin{cor}\label{cor:ricciflat}
Assume $M$ is K\"ahler Ricci-flat and that $L$ is either (i) minimal Lagrangian, or (ii) totally real and $J$-minimal, or (iii)
totally real with first Betti number $b^1(L)=0$. Then $L$ is Maslov-zero.
\end{cor}

Situation (i) includes the case of special Lagrangian submanifolds in Calabi-Yau manifolds. We thus obtain a new proof of the well-known fact that special Lagrangians are Maslov-zero: the standard proof relies on the topological definition of the Maslov index in terms of the twisting of $\Lambda^nTL$ with respect to the parallel section of $K_M$ determined by the Calabi-Yau condition.

\paragraph{Monotonicity of minimal Lagrangian submanifolds.} As mentioned in Section \ref{s:prelim}, in the symplectic context it is interesting to compare the Maslov index with $\int_\Sigma\omega$, the \textit{symplectic area}. This is best understood from the cohomological point of view introduced at the beginning of Section \ref{s:applications}. Indeed, for Lagrangian boundary data it is simple to check that the symplectic areas of two surfaces $\Sigma_1$, $\Sigma_2$ coincide if they belong to the same relative homology class. Symplectic area can thus also be viewed as a relative cohomology class $\alpha_L$, and our problem can be phrased in terms of comparing $\mu_L$ to $\alpha_L$ in $H^2(M,L)$.

Consider the long exact sequence
\begin{equation}\label{eq:longexact}
 \dots \rightarrow H^1(L) \stackrel{\delta}{\rightarrow} H^2(M,L) \stackrel{j}{\rightarrow} H^2(M) \rightarrow\dots
\end{equation}
The fact that the Maslov index is twice the first Chern number when $\Sigma$ has empty boundary shows that the image of $\mu_L$ in $H^2(M)$ is $2c_1(M)$. The image of $\alpha_L$ in $H^2(M)$ is $[\omega]$.

Assume $M$ is K\"ahler-Einstein, \textit{i.e.} $\rho=c\,\omega$ for some $c\in \R$. Then $2\pi c_1(M)=c[\omega]$ so $j(\pi\mu_L-c\,\alpha_L)=0\in H^2(M)$. Exactness implies that there must exist some $\beta\in H^1(L)$ such that $\delta(\beta)=\pi\mu_L-c\,\alpha_L$. We can find $\beta$ using our integral formula, as follows.

Corollary \ref{cor:kahler_formula} shows that
\begin{equation*}
 \pi\,\mu_L(\Sigma,\partial\Sigma)-c\int_\Sigma\omega=-\int_{\partial\Sigma}\omega(H,\cdot).
\end{equation*}
Equation (\ref{eq:K_fundamental}) and the K\"ahler-Einstein hypothesis imply that, for $L$ Lagrangian, $\omega(H,\cdot)$ is a closed 1-form on $L$ so it defines a cohomology class in $H^1(L)$. We thus obtain the main result of \cite{CG}.

\begin{cor}
 Assume $M$ is K\"ahler-Einstein with $\rho=c\,\omega$ and $L$ is Lagrangian. Then
 \begin{equation*}
 \pi\mu_L-c\,\alpha_L=-[\omega(H,\cdot)].
\end{equation*}
In particular, any minimal Lagrangian $L$ is monotone.
 \end{cor}
One should compare this with Corollary \ref{cor:definite_ricci} which produces a definite sign for the Maslov index only for complex curves, but with less stringent hypotheses.

\paragraph{Monotonicity of $J$-minimal submanifolds.} As discussed in Section \ref{s:immersed}, minimal Lagrangians make sense mostly in the context of K\"ahler-Einstein manifolds. In \cite{LP3} it is shown that, in the more general context of K\"ahler manifolds with (positive or negative) definite Ricci curvature, $J$-minimal submanifolds provide an interesting substitute. 

In this context $\rho$ defines a second symplectic form on $M$ and $\pm \rho(\cdot,J,\cdot)$ is a Riemannian metric. We will say that a submanifold is \textit{$\rho$-Lagrangian} if $\rho_{|TL}\equiv 0$. Any such $L$ is totally real. All standard results in Symplectic Geometry apply to $\rho$ and to $\rho$-Lagrangians, \textit{e.g.} the energy of holomorphic curves with boundary on $\rho$-Lagrangians is topologically determined by the homology class, thus bounded.

Given any surface $\Sigma$ with boundary on a $\rho$-Lagrangian $L$, the analogue of the symplectic area is the quantity
\begin{equation}
 \alpha_L(\Sigma,\partial\Sigma):=\int_{\Sigma}\rho.
\end{equation}
As usual this defines a class $\alpha_L\in H^2(M,L)$. We will say that $L$ is \textit{$\rho$-monotone} if $\mu_L$ is proportional to $\alpha_L$. 

Our main reason for interest in this notion is the fact that any $J$-minimal submanifold is $\rho$-Lagrangian: this is immediate from the definitions and from Equation (\ref{eq:K_fundamental}). In the long exact sequence (\ref{eq:longexact}), $\alpha_L$ has image $2\pi c_1(M)\in H^2(M)$ so $j(\pi\mu_L-\alpha_L)=0$. We then obtain the precise description of a class $\beta\in H^1(L)$, as above. This leads to the following result which further reinforces the analogies between $J$-minimal submanifolds and minimal Lagrangians.

\begin{cor} 
Assume $M$ is K\"ahler with (positive or negative) definite Ricci curvature. Then any $J$-minimal submanifold is $\rho$-monotone.
\end{cor}

\paragraph{Monotonicity of solitons.} Recall that an immersed Lagrangian submanifold $\iota:L\rightarrow\C^n$ is a \textit{self-similar soliton} if, for some $c\in\R$ and all $x\in L$, $H(x)=c\,\iota(x)^\perp$. In this case, under mean curvature flow and up to reparametrization, $\iota$ evolves simply by rescaling: $\iota(t)=\rho(t)\cdot\iota$. Thus $H(t,x)=\frac{1}{\rho(t)}H(x)$ and each $\iota(t)$ is a soliton, with $c(t)=\frac{1}{\rho^2(t)}$. 

Using the Lagrangian hypothesis the soliton equation can be re-written as an equation of 1-forms on $L$:
\begin{equation*}
 \iota^*\omega(H,\cdot)=c\,\iota^*\omega(\iota,\cdot)=c\,\iota^*g(J\iota,\cdot).
\end{equation*}
Setting $\lambda:=\frac{1}{2}(xdy-ydx)$ and $\iota=(x,y)$ we find that the soliton equation is equivalent to $\iota^*\omega(H,\cdot)=2c\,\iota^*\lambda$. Now notice that $\omega=d\lambda$. Corollary \ref{cor:kahler_formula} then leads to the following fact, cf. \cite{GSSZ}.
\begin{cor}\label{cor:solitons}
 Let $L$ be a Lagrangian soliton in $\C^n$. Then $\mu_L=-\frac{2c}{\pi}\alpha_L$, so $L$ is monotone.
\end{cor}
Notice that $\mu_L$ is integral and $\alpha_L$ rescales by $\rho^2(t)$: this is consistent with the rescaling of $c(t)$.

\bibliographystyle{amsplain}
\bibliography{MCW}

\end{document}